\documentclass[11pt,a4paper]{article}

\usepackage{amsmath,amsthm,amssymb,a4wide}
\usepackage{times}
\usepackage{enumerate}

\usepackage{setspace}

\usepackage{dsfont}

\newcommand{\R}{\mathbb R}

\theoremstyle{definition}
\newtheorem{thm}{Theorem}[section]
\newtheorem{cor}[thm]{Corollary}
\newtheorem{prop}[thm]{Proposition}
\newtheorem{lem}[thm]{Lemma}

\newtheorem{rem}[thm]{Remark}

\newtheorem{exa}[thm]{Example}

\newcommand{\subjclass}[1]{\bigskip\noindent\emph{2010 Mathematics Subject Classification:}\enspace#1}

\numberwithin{equation}{section}

\begin{document}

\title{A degenerate elliptic equation for second order\\ controllability of nonlinear systems}
\author{Pierpaolo Soravia\thanks{email: soravia@math.unipd.it.  }\\
Dipartimento di Matematica\\ Universit\`{a} di Padova, via Trieste 63, 35121 Padova, Italy}

\date{}
\maketitle

\begin{abstract}
For a general nonlinear control system, we study the problem of small time local attainability of a target which is the closure of an open set. When the target is smooth and locally the sublevel set of a smooth function, we develop second order attainability conditions as explicit pointwise conditions on the vector fields at points where all the available vector fields are contained in the tangent space of its boundary. Our sufficient condition requires the function defining the target to be a strict supersolution of a second order degenerate elliptic equation and if satisfied, it allows to reach the target with a piecewise constant control with at most one switch. For symmetric systems, our sufficient condition is also necessary and can be reformulated as a suitable symmetric matrix having a negative eigenvalue. For nonlinear affine systems to obtain a necessary and sufficient condition we require an additional request on the drift.
Our second order pde has the same role of the Hamilton-Jacobi equation for first order sufficient conditions.
\end{abstract}

\subjclass{Primary 49L20; Secondary 93B05, 35F21, 35D40.}

\section{Introduction}

In this paper we study a general controlled nonlinear system
\begin{equation}\label{eqsys}
\left\{\begin{array}{ll}\dot x_t=f(x_t,a_t),\\
x_0\in {\mathbb R}^n,
\end{array}\right.
\end{equation}
where $f:\R^n\times A\to\R^{n}$ is continuous and determines uniqueness of trajectories, $A\subset \R^m$ is compact and $a.:[0,+\infty)\to A$ will always be piecewise continuous.
We are also given a function $u:\R^n\to\R$, $u\in C^2(\R^n)$ and a point $\bar x\neq0$, $\nabla u(\bar x)\neq 0$.  The idea is that $u$ describes locally around $\bar x$ the boundary of a target set for (\ref{eqsys}), i.e. locally 
\begin{equation}\label{eqtarget}
{\mathcal T}=\{x:u(x)\leq u(\bar x)\}.
\end{equation}
We want to study the problem of small time local attainability (STLA) of the target ${\mathcal T}$ by system (\ref{eqsys}) at $\bar x$.
In other words, given any small time $t>0$, study when the point $\bar x$ is in the interior of the set of points from which we can reach $\mathcal T$ in time less than $t$ following trajectories of (\ref{eqsys}) or equivalently, the minimum time function is continuous at $\bar x$, see the book by Bardi and Capuzzo-Dolcetta \cite{bcd}. Here, function $u$ may for instance be the signed distance function from the hypersurface $M=\{x:u=u(\bar x)\}$ but this is not required.
The STLA problem is a classical and important subject in control theory. It is well known that when the target is STLA for (\ref{eqsys}) at every point on its boundary, then the minimum time function is continuous in its whole domain, and we can characterise it as the unique solution of a free boundary problem for the Hamilton-Jacobi equation, see Bardi and the author \cite{baso2} and the references therein. 
Moreover the rate of decay of the minimum time function at the target determines the exponent of the local H\"older continuity of such function in its domain, see \cite{bcd}. 
Therefore STLA is a way we approach the regularity of the minimum time function, which is the prototype of solutions of Dirichlet boundary value problems for the Hamilton Jacobi equation. For the connection, see for instance the recent paper by Bardi, Feleqi and the author \cite{bfs}. It is then clear that what we derive here has precise consequences on that problem as well.

Sufficient conditions for STLA at a point of system (\ref{eqsys}) appear in the literature in different ways. The most interesting are pointwise conditions on the controlled vector field $f$ and the function $u$ that locally describes the target. Sufficient conditions have different nature. Classical, explicit first order attainability conditions require that for some $a\in A$ the vector field $f(\cdot,a)$ points inward the target, namely $\nabla u\cdot f(\bar x,a)<0$ (Petrov condition, see e.g. \cite{pe2}). Notice at this point that this is equivalent to asking that $u$ (or a multiple of $u$) satisfies in a neighborhood of $\bar x$
\begin{equation}\label{eqpetrov}
\max_{a\in A}\{-f(x,a)\cdot \nabla u(x)\}\geq1,\end{equation}
namely $u$ is a supersolution of the classical first order Hamilton-Jacobi equation satisfied by the minimum time function.
A consequence is that the minimum time function $T$ satisfies a local estimate of the form
\begin{equation}\label{eqfirst}
T(x)\leq C d(x),\quad x\in B_r(\bar x),
\end{equation}
where $d(x)$ is the distance function from the target, see Bardi-Falcone \cite{BaFa90}.
If Petrov condition fails, namely if $\nabla u\cdot f(\bar x,a)\geq0$ for all $a\in A$, one can give second order conditions. If explicit, second order conditions, at least for symmetric systems, require that a Lie bracket between two available vector fields of the system is transversal to $\partial\mathcal T$ at $\bar x$.
In this case (\ref{eqfirst}) becomes
\begin{equation}\label{eqsecond}
T(x)\leq C d^{1/2}(x),\quad x\in B_r(\bar x),
\end{equation}
showing the difference between first and second order conditions.
For systems that are not symmetric, as we see in the following, the situation is not completely clarified. Let us discuss the transversality of a Lie bracket in more detail. There are at least two unpleasant facts in requiring a transversal Lie bracket as a second order sufficient condition for STLA.
The first one is that constructing a Lie bracket requires two vector fields, while sometimes, as we see in some explicit example below, we can reach the target using a single vector field (constant control) when the geometry of the manifold (in particular its curvature) is favourable as compared to the trajectory, and still have (\ref{eqsecond}) satisfied.
The second and more important for applications is that in order to construct a trajectory that follows (with an error) the new vector field provided by a Lie bracket of two vector fields, one needs to build a trajectory
that has three switches, and they need to happen in short time, if we want to keep the error small. Indeed in order for the system to follow the vector field $[f(\cdot,a_1),f(\cdot,a_2)]$, which denotes the Lie bracket between the two vector fields, for small time $t>0$ we need to use the control
\begin{equation}\label{eqbracket}
a_s=\left\{
\begin{array}{llll}
a_1,\quad& s\in[0,t),   \\a_2,\quad& s\in[t,2t)   \\  
-a_1, &s\in[2t,3t),\\
-a_2,& s\in[3t,4t).
\end{array}
\right.\end{equation}
Time needs to be short because we still get an error of size $t^3$ when $f(\cdot,a) \in C^2$ for $a=a_1,a_2$. 
Using the control in (\ref{eqbracket}) is possible provided the system is symmetric, i.e. for instance $f(\cdot,-a)=-f(\cdot,a)$ for $a=a_1,a_2$, otherwise for generic nonlinear systems the use of Lie brackets appears less direct.

In this paper we aim at simplifying the treatment of second order sufficient conditions for STLA of the system (\ref{eqsys}) at $\bar x$. We want to derive explicit conditions on the vector field that are easy to check and as close as possible to the known necessary conditions.
We will see that if the system is not symmetric, we need to expect that (\ref{eqsecond}) is relaxed to the weaker
\begin{equation}\label{eqsecondb}
T(x)\leq C |x-\bar x|^{1/2},\quad x\in B_r(\bar x).
\end{equation}
The continuity of the minimum time function at $\bar x$ is still satisfied, and this estimate can improve to (\ref{eqsecond}) if either
Petrov or a second order condition is satisfied at every point of the boundary of the target in the neighborhood of $\bar x$ and other appropriate conditions hold. Then by a standard mechanism we still get local 1/2-H\"older regularity of the minimum time function in its domain, see e.g. \cite{so}, or Theorem IV.1.18 in (\cite{bcd}). 

In order to achieve our goal, we start from the conjecture that if a smooth target is STLA at $\bar x$ for the system, then starting at each point in the neighborhood of $\bar x$, we can reach the target using piecewise constant controls with at most one switch.
We will actually prove this conjecture for symmetric systems where our conditions will be necessary and sufficient for STLA, and for nonlinear-affine systems where $f(x,a)=\sigma_0(x)+\sigma(x)a$, under an additional request on the drift $\sigma_0$, while it remains open for general ones. For nonlinear-affine systems, our approach leads us to sufficient conditions that have the same nature (even less restrictive) of those pointwise and explicit appearing in the literature that we are aware of, but we believe that they are presented in a more natural and direct way. For nonlinear affine systems our conditions can be equivalently rewritten in an algebraic fashion making them quite easy to check.
As far as we know, we develop new explicit conditions for general nonlinear systems. 

To be more specific, we introduce the following second order degenerate elliptic inequality in a neighborhood of $\bar x$
\begin{equation}\label{eqdegenerate}
\begin{array}{ll}
\max_{(a_1,a_2)\in A\times A}\{ -{\mbox Tr}(D^2u\; f(x,a_1)\otimes f(x,a_2))\\
-(D(f(x,a_1)+f(x,a_2))(f(x,a_1)+f(x,a_2))
+[f(\cdot,a_1),f(\cdot,a_2)](x))\cdot\nabla u(x)\} \geq1
\end{array}\end{equation}
to be satisfied by $u$ in (\ref{eqtarget}) (or by a multiple of $u$),
and show that if $f(\bar x,A)$ is contained in the tangent space of the target at $\bar x$, then $\mathcal T$ is STLA at $\bar x$ for the system and (\ref{eqsecondb}) holds.
Thus if $u$ is a supersolution of the corresponding equation, then the system satisfies second order STLA at $\bar x$ making (\ref{eqdegenerate}) a counterpart for second order conditions of the Hamilton-Jacobi inequality (\ref{eqpetrov}) when Petrov condition holds.
We remark that (\ref{eqdegenerate}) always contains the case of a target reachable by a single tangent vector field (when the inequality is attained at $a_1=a_2$) and, in the case of symmetric systems, the case of a transversal Lie bracket. Therefore, for $f(x,a)=\sigma(x)a$, (\ref{eqdegenerate}) is more general than the classical one on transversality of a Lie bracket and can be easily reformulated by checking if a suitable symmetric matrix has a negative eigenvalue. Our approach is constructive, since the corresponding eigenvector contains the coordinates of the two controls that we can use to define a trajectory reaching the target with at most one control switch. The case of nonlinear-affine systems can be equivalently reformulated with a rather simple algebraic condition as well.
Inequality (\ref{eqdegenerate}) involves iterated first order hamiltonians, as we describe in the next section, making it a second order operator. The general idea shows how we could proceed and deal with higher order sufficient conditions as well, although we do not do it here.
Part of our results on symmetric systems, has been announced in the proceedings \cite{so1}.

Small time local attainability and regularity of the minimum time function is a long studied and important subject in optimal control.
Besides classical results by Kalman (for linear systems) and by Sussman, who mainly deal with controllability at equilibrium points of the system, we recall Petrov \cite{pe,pe2} for the study of first order controllability, that is attainability at a single point. Liverowskii \cite{li} studied the corresponding problem of second order, see also Bianchini and Stefani \cite{bs,bs2,bs3}. Controllability of higher order to a point was studied by Liverowskii \cite{li2}. For attainability of a target different from a point we recall the papers by Bacciotti \cite{ba} in the case of targets of codimension 1 and the author \cite{so} for manifolds of any dimension, possibly with a boundary. We mention Bardi-Falcone \cite{BaFa90} who showed that Petrov condition is also necessary for local Lipschitz continuity of the minimum time function, while in \cite{bfs} we derive necessary second order conditions.
More recently the work by Krastanov and Quincampoix \cite{kr,kr2} pointed out the importance of the geometry of the target and studied higher order attainability of nonsmooth targets for affine systems with nontrivial drift.  For the same class of systems Marigonda, Rigo and Le \cite{ma,ma2,ma3} studied higher order regularity focusing on the lack of smoothness of the target and the presence of state constraints. We finally mention the paper by Motta and Rampazzo \cite{mr} where the authors study higher order hamiltonians obtained by adding iterated Lie brackets as additional vector fields, in order to prove global asymptotic controllability to a target. Their higher order Hamiltonian is still a first order operator in contrast to ours.

In the following, $a\cdot b$ indicates the scalar product in $\R^n$, and ${}^tB$ the transpose of a matrix $B$. Notation $B_r(x)$ indicates the closed ball of radius $r>0$.
If $f:\R^n\times A\to\R^n$ is a smooth controlled vector field, in components $f=(f_i)_{i=1,\dots,n}$, we will denote $Df(x)=\left(\partial_{x_j}f_i\right)_{i,j=1\dots,n}$ the spatial jacobian matrix of $f$ at $x$. 
We say that the control vector field $f\in C(\R^n\times A;\R^n)$ is of class $C^1$ if all its spatial partial derivatives $\partial_{x_i}f\in C(\R^n\times A;\R^n)$ so they are jointly continuous in all variables. The same notation will be assumed for higher order derivatives as well.
As a general rule, in the product of functions having the same dependance, we will show their argument only after the last factor. Everything we develop will be in an appropriate neighborhood of a given point, but we keep all functions everywhere defined for convenience.

The contents of the paper are as follows. In section 2 we derive some preliminaries and introduce iterated hamiltonians. In section 3 we present out sufficient condition and derive STLA and estimates on the minimum time function. In section 4 we show that our sufficient condition is necessary for symmetric systems and we partly extend such result in the case of nonlinear affine systems in section 5. Finally in section 6 we show some examples illustrating our ideas.

\section{Preliminaries}

In this section we recover the main estimates that we are going to use. We start by deriving some known Taylor expansions on trajectories. We analyse the trajectory resulting from a unique switch between two smooth vector fields.
For vector fields $f,g\in C^1(\R^n;\R^n)$, we will use below the notation of Lie bracket as the vector field
$$[f,g]:\R^n\to\R^n,\quad [f,g](x)=Dg\;f(x)-Df\;g(x).$$

\begin{prop}\label{prop1}
Let $t>0$ and $f,g:\R^n\to\R^n$ be $C^1$ vector fields. 
Consider the Caratheodory solution of
\begin{equation}\label{eqswitch}
\dot x_s=\left\{\begin{array}{ll}f(x_s),\quad&\mbox{ if }s\in[0,t),\\
g(x_s),&\mbox{ if }s\in[t,2t],\end{array}\right.
\quad x_0\in\R^n.
\end{equation}
Then, as $t\to0+$,
$$\begin{array}{l}
x_{2t}-x_0
=(f(x_0)+g(x_0))t+D(f+g)\;(f+g)(x_{0})\frac{t^2}2
+[f,g](x_0)\frac{t^2}2+o(t^2).
\end{array}$$
\end{prop}

\begin{proof} Observe that
$$\begin{array}{ll}
x_{2t}-x_0=x_{2t}-x_t+x_t-x_0
=\int_0^t(g(x_{t+s})+f(x_s))\;ds\\
=\int_0^t[(g(x_{t+s})-g(x_t))+(f(x_s)-f(x_0)]\;ds
+g(x_t)t+f(x_0)t\\
=\int_0^t[\int_0^s(Dg\;g(x_{t+r})+Df\;f(x_r))\;dr]\;ds
+\int_0^t(g(x_s)+sDg\;f(x_s))\;ds+f(x_0)t\\
=\int_0^t[\int_0^s(Dg\;g(x_{t+r})+Df\;f(x_r)+Dg\;f(x_r))\;dr]\;ds
+\int_0^tsDg\;f(x_s)\;ds+(f(x_0)+g(x_0))t\\
=(Dg\;g(x_{0})+Df\;f(x_0)+2Dg\;f(x_0))\frac{t^2}2
+(f(x_0)+g(x_0))t+o(t^2).
\end{array}
$$
\end{proof}
The previous result shows the main trajectories that we are going to use in our analysis in order to prove that a system is STLA at a point. A sort of balanced trajectories with a single switch. Observe that if we modify $t$ then the trajectory changes.
\begin{rem} Note that the Taylor estimate on the trajectory of a single vector field shows that if we consider the averaged system
$$\left\{\begin{array}{ll}
\dot y_s=\frac{f(y_s)+g(y_s)}2,\quad&\mbox{ for }s\in[0,2t],\\
y_0=x_0,
\end{array}\right.$$
then it satisfies, as $t\to0+$,
$$\begin{array}{l}y_{2t}-x_0
=(f(x_0)+g(x_0))t+D(f+g)\;(f+g)(x_{0})\frac{t^2}2
+o(t^2).
\end{array}$$
Therefore from Proposition \ref{prop1} we conclude that
$$x_{2t}=y_{2t}+[f,g](x_0)\frac{t^2}2+o(t^2),\quad \hbox{as }t\to0+.$$
It means that one switch between two admissible vector fields causes a deflection from the trajectory of the average of the vector fields by a second order term proportional to their Lie bracket.
This is a special case of the well known Baker-Campbell-Hausdorff formula stopped at the second order.
In the above statements all remainders $o(t^2)$ become terms of the order of $t^3$ if the vector fields $f,g\in C^2$.
\end{rem}
The following straightforward result is useful to localise the problem. Here we go back to a controlled vector field $f\in C(\R^n\times A;\R^n)$, $A$ compact.
\begin{lem}\label{lemesttraj}
Let $\bar x\in\R^n$. Let $0<\delta<r$ and $M=\max_{(x,a)\in B_r(\bar x)\times A}|f(x,a)|$.
Then for all $x_o\in B_\delta(\bar x)$ the trajectory of (\ref{eqsys}) satisfies $x_t\in B_r(\bar x)$ for all $t\in [0,(r-\delta)/M]$.
\end{lem}
\begin{proof}
Since the trajectory of (\ref{eqsys}) satisfies,
$$x_t=x_o+\int_0^tf(x_s,a_s)\;ds,$$
then for $t\in [0,(r-\delta)/M]$, we get $|x_t-\bar x|\leq|x_t-x_o|+\delta\leq Mt+\delta\leq r$.
\end{proof}

\subsection{Second order Hamiltonians}

For a given vector field $f\in C^1(\R^n;\R^n)$, we define as usual its Hamiltonian as the functional
$$H_f:C^1(\R^n)\to C(\R^n),\quad H_fu=-\nabla u\cdot f.$$
Given two vector fields $f,g\in C^1(\R^n;\R^n)$, we introduce the second order hamiltonian $H_{g,f}:C^2(\R^n)\to C(\R^n)$ by
$$H_{g,f}u=H_g\circ H_fu=\nabla(\nabla u\cdot f)\cdot g=D^2u\;f\cdot g+\nabla u\cdot Df\;g.$$
Notice that the second order Hamiltonian is therefore the following degenerate elliptic operator for $u\in C^2(\R^n)$
$$H_{g,f}u=\mbox{Tr}(D^2u\;f\otimes g)+\nabla u\cdot Df\;g.$$
As a differential operator, the second order hamiltonian is bilinear in its arguments.
We could also define higher order Hamiltonians although we will not consider them in this paper.

By Taylor expansions of functions of multiple variables we then obtain second order estimates of the variation of functions along trajectories of (\ref{eqswitch}) as a consequence of Proposition \ref{prop1}, expressed by second order hamiltonians.
\begin{prop}\label{prop2}
Let $t>0$ and $f,g:\R^n\to\R^n$ be $C^1$ vector fields. Let $u:\R^n\to\R$ be a function of class $C^2$. For the trajectory (\ref{eqswitch}) we have the following estimate, as $t\to0+$,
\begin{equation}\label{eqest}\begin{array}{ll}
u(x_{2t})-u(x_0)&=\nabla u\cdot(f+g)(x_0)t+(H_{{f+g},{f+g}}u(x_0)+\nabla u\cdot[f,g](x_0))\frac{t^2}2+o(t^2)\\
&=\nabla u\cdot(f+g)(x_0)t
+(H_{f,f}u(x_0)+H_{g,g}u(x_0)+2H_{f,g}u(x_0))\frac{t^2}2+o(t^2).
\end{array}\end{equation}
If in particular $f\equiv g$ then
$$u(x_{2t})-u(x_0)=2(\nabla u\cdot f(x_0)t+H_{f,f}u(x_0){t^2})+o(t^2).$$
\end{prop} 
\begin{proof}
From the standard Taylor estimate and Proposition \ref{prop1}
$$\begin{array}{ll}
u(x_{2t})-u(x_0)&=\nabla u\cdot(f+g)(x_0)t
+\nabla u\cdot(D(f+g)\;(f+g)(x_0)+[f,g](x_0))\frac{t^2}2\\&
+\frac12D^2u(f+g)\cdot(f+g)t^2+o(t^2)
\end{array}$$
from which the first equality in (\ref{eqest}) follows. The second equality is a consequence of the simple observation that
$$\nabla u\cdot[f,g]=H_{f,g}u-H_{g,f}u,$$
when $u\in C^2$, and by bilinearity of the second order hamiltonian.
\end{proof}
\begin{rem}
If the vector fields $f,g$ are at least of class $C^2$ and the function $u$ is at least of class $C^3$, then the remainders in Proposition \ref{prop2} are of the order $t^3$. Notice that in (\ref{eqest}) we obtain the variation of $u(x_\cdot)$ at the point $2t$ after a complete switch, and modifying $t$ will change the trajectory we are following.
\end{rem}

\section{Sufficient conditions for second order controllability}

This section contains a key general result of the paper, that we will later investigate and specialise to more restrictive classes of systems.
We want to discuss when a given target, locally described in the neighborhood of a point $\bar x$ as $\mathcal T=\{x:u(x)\leq u(\bar x)\}$, is STLA at $\bar x$ by system (\ref{eqsys}). We plan to reach the target only using trajectories with a single (at most) switch between two available vector fields of the system, in the form of (\ref{eqswitch}).
We recall the definition of minimum time function for system (\ref{eqsys}) and target $\mathcal T$, namely
$$T(x)=\inf_{a_\cdot\in L^\infty(0,+\infty)}t_x(a_\cdot),$$
where $t_x(a_\cdot)=\inf\{t\geq0:x_t\in\mathcal T,\;x_\cdot \mbox{ trajectory of }(\ref{eqsys})\}$. 
In particular $T\equiv0$ on $\mathcal T$.
The following result is a consequence of Proposition \ref{prop2}.
\begin{thm}\label{thm1}
Let $f:\R^n\times A\to\R^{n}$ be of class $C^2$ and let $u:\R^n\to\R$ be a function of class $C^3$. Let $\bar x\in\R^n$ be a point such that $\nabla u(\bar x)\neq0$, $\nabla u\cdot f(\bar x,a)=0$, for all $a\in A$ and suppose that
\begin{equation}\label{eqpde}
\begin{array}{ll}
\max_{(a_1,a_2)\in A\times A}\{ -{\mbox Tr}(D^2u(\bar x)\; f(\bar x,a_1)\otimes f(\bar x,a_2))\\
-(D(f(\bar x,a_1)+f(\bar x,a_2))(f(\bar x,a_1)+f(\bar x,a_2))+[f(\cdot,a_1),f(\cdot,a_2)](\bar x))\cdot\nabla u(\bar x)\}>0.
\end{array}\end{equation}
Then the target $\{x:u(x)\leq u(\bar x)\}$ is STLA for the system (\ref{eqsys}) at $\bar x$ and the minimum time function $T$ satisfies in the neighborhood of $\bar x$
\begin{equation}\label{eqestimate}
T(x)\leq C|x-\bar x|^{1/2}.\end{equation}
\end{thm}
\begin{rem}
If (\ref{eqpde}) holds, then locally around $\bar x$ the inequality is preserved by continuity. Therefore (\ref{eqpde}) holds if and only if $u$ is a strict supersolution of the corresponding elliptic partial differential equation in a neighborhood of $\bar x$
\begin{equation}
\begin{array}{ll}
\max_{(a_1,a_2)\in A\times A}\{ -{\mbox Tr}(D^2u(x)\; f(x,a_1)\otimes f(x,a_2))
-(D(f(x,a_1)+f(x,a_2))(f(x,a_1)+f(x,a_2))\\
\quad+[f(\cdot,a_1),f(\cdot,a_2)](x))\cdot\nabla u(x)\}
\geq \rho>0,\quad x\in B_r(\bar x),
\end{array}\end{equation}
for some $r,\rho>0$.
\end{rem}
\begin{proof}
Assume (\ref{eqpde}) and let $a_1,a_2\in A$ be such that
\begin{equation}\label{eqpdepoint}
\begin{array}{ll}-{\mbox Tr}(D^2u\; f(\bar x,a_1)\otimes f(\bar x,a_2))
-(D(f(\bar x,a_1)+f(\bar x,a_2))(f(\bar x,a_1)+f(\bar x,a_2))\\
\quad +[f(\cdot,a_1),f(\cdot,a_2)](\bar x))\cdot\nabla u(\bar x)>0.\end{array}
\end{equation}
By continuity there are $\rho,r>0$ such that
\begin{equation}\label{eq2ndest}\begin{array}{ll}-{\mbox Tr}(D^2u\; f(x,a_1)\otimes f(x,a_2))
-(D(f(x,a_1)+f(x,a_2))(f(x,a_1)+f(x,a_2))\\
\quad +[f(\cdot,a_1),f(\cdot,a_2)](x))\cdot\nabla u(x)>\rho,\quad x\in B_r(\bar x).\end{array}
\end{equation}
Let $M=\max\{|f(x,a)|:(x,a)\in B_r(\bar x)\times A\}$.
Let $f(\cdot)=f(\cdot,a_1),\;g(\cdot)=f(\cdot,a_2)$, 
we follow the trajectory (\ref{eqswitch}) of the two vector fields such that $\nabla u\cdot f(\bar x)=0=\nabla u\cdot g(\bar x)$, starting out at $x_0=\bar x$. Therefore by (\ref{eqest}), and (\ref{eq2ndest}),
$$u(x_{2t})-u(\bar x)\leq-\rho\frac{t^2}2+Ct^3,$$
for $t$ positive and sufficiently small so that the trajectory remains in $B_r(\bar x)$, and for some positive constant $C$ estimating the remainder term of (\ref{eqest}) in $B_r(\bar x)$. We want to keep $t\leq \rho/{4C}$ so that the right hand side remains negative.

We now use continuous dependence on the initial condition and start the trajectory $x^1_\cdot$ in (\ref{eqswitch}) from a point $x_0=x^1$, $|x^1-\bar x|\leq\delta<r$, $\delta$ to be decided later. We fix any 
\begin{equation}\label{eqtimeest}
0< t\leq \min\{\rho/{4C},(r-\delta)/M,1\}\end{equation}
and obtain, by Lemma \ref{lemesttraj} and (\ref{eqest}),
\begin{equation}\label{eq2nd}
\begin{array}{ll}
u(x^1_{2t})-u(\bar x)&
=u(x^1_{2t})-u(x^1)+u(x^1)-u(x_p)
\leq L_1|x^1-\bar x| t-\rho \frac{t^2}2+Ct^3+Ld(x^1)\\
&\leq L_1|x^1-\bar x| t-\frac{\rho}4t^2+Ld(x^1)
\leq  \tilde L|x^1-\bar x|-\frac{\rho}4t^2,
\end{array}\end{equation}
where $L$ is a local Lipschitz constant for $u$, $L_1$ is a local Lipschitz constant for the product $\nabla u(x)\cdot (f(x)+g(x))$ which vanishes at $\bar x$ by the assumption, $\tilde L=L+L_1$, and $x_p\in\partial {\mathcal T}$ is such that $d(x^1)=|x^1-x_p|$. The point $x_p\in B_r(\bar x)$ for $\delta\leq r/2$.
Note that the right hand side of (\ref{eq2nd}) is zero at
$$\bar t=2\sqrt{\frac{\tilde L}{\rho}}\;|x^1-\bar x|^{1/2}\leq 2\sqrt{\frac{\tilde L\delta}{\rho}},$$
and that the right hand side of the previous formula is smaller than the right hand side of (\ref{eqtimeest}) for $\delta>0$ sufficiently small.
Therefore the trajectory (\ref{eqswitch}) starting at any $x^1\in B_\delta(\bar x)$ will reach the target $\{x:u(x)\leq u(\bar x)\}$ earlier than $\bar t$. In particular we can estimate the minimum time to reach the target as
$$T(x^1)\leq 2\sqrt{\frac{\tilde L}{\rho}}\;|x^1-\bar x|^{1/2},\quad x^1\in B_\delta(\bar x),$$
namely with the square root of the distance from the centre of the ball on the target.
Hence $T$ is continuous at $\bar x$ and the system is STLA at $\bar x$.
\end{proof}

\begin{rem}\label{remcontrollability}
Notice that the same conclusion of Theorem \ref{thm1} will hold assuming directly that there exist $a_1,a_2\in A$ such that $\nabla u(\bar x)\cdot (f(\bar x,a_1)+f(\bar x,a_2))=0$ and (\ref{eqpdepoint}) holds, a little weaker than the statement. However, if one of the two vector fields points inward $\mathcal T$, we could reach a stronger conclusion with a similar argument. Indeed, if we know that for $r,\rho>0$ and $a\in A$
\begin{equation}\label{eq1storder}
-\nabla u\cdot f(x,a)\geq \rho>0,\quad x\in B_r(\bar x),
\end{equation}
and  $(x_t)_{t\geq 0}$ is the trajectory of the vector field $f(\cdot,a)$ with initial point $x_o=\bar x$, we easily obtain an estimate of the form
$$u(x_t)-u(\bar x)\leq -\rho t+Ct^2,$$
for $t$ sufficiently small. Here the leading negative term has a first order power in $t$. If now $|x^1-\bar x|\leq\delta$ and we follow the trajectory of $f(\cdot,a)$ starting at $x^1$, call it $(x^1_t)_{t\geq0}$, then the estimate (\ref{eq2nd}) becomes
\begin{equation}\label{eq1st}
u(x^1_t)-u(\bar x)=u(x^1_t)-u(x^1)+u(x^1)-u(x_p)\leq -\rho t+Ct^2+Ld(x^1),\end{equation}
where $L$ is a local Lipschitz constant for $u$ and $x_p\in\partial\mathcal T$ is such that $d(x^1)=|x^1-x_p|$. It follows from here that for $\delta$ sufficiently small, the minimum time to reach the target from $x_1$ can be estimated as
\begin{equation}\label{eqnew}
T(x^1)\leq\frac{2L}\rho d(x^1),\quad x^1\in B_\delta(\bar x),\end{equation}
therefore with the distance of $x^1$ from the target. Thus the target is STLA for the system at $\bar x$  but the minimum time function satisfies a stronger estimate. Moreover notice that, in contrast with the previous calculation, from (\ref{eqestimate}) it is yet unclear how to pass directly to (\ref{eqsecond}) due to the presence of a nonvanishing first order term in (\ref{eqest}) at $x^1$.
\end{rem}

In order to improve estimate (\ref{eqestimate}) to (\ref{eqsecond}), the question turns out to be quite clear for symmetric systems, see next Corollary \ref{corsecondest} and Remark \ref{remextendedsym}. Less definitive results can be provided for general systems. The question has a delicate point: if the system is not symmetric, a second order sufficient attainability condition valid at a point is not preserved at all points in the neighborhood and this may even lead to discontinuities of the minimum time function, as we see in Example \ref{ex3} below. The reason is that if the vector fields are tangent at a point, they need not be tangent in the neighborhood.
To overcome this difficulty, we need to add assumptions and we can either improve the proof of Theorem \ref{thm1} or require a controllability condition at each point in the neighborhood of $\bar x$.
\begin{cor}\label{corsecondest}
Let $f:\R^n\times A\to\R^{n}$ be of class $C^2$ and let $u:\R^n\to\R$ be a function of class $C^3$. Let $\bar x\in\R^n$ and $r,\rho>0$, $a_1,a_2\in A$ be such that $\nabla u\cdot (f(\bar x,a_1)+f(\bar x,a_2))=0$, and
\begin{equation}\label{eqcor}
\begin{array}{ll}
-{\mbox Tr}(D^2u(x)\; f(x,a_1)\otimes f(x,a_2))
-(D(f(x,a_1)+f(x,a_2))(f(x,a_1)+f(x,a_2))\\
+[f(\cdot,a_1),f(\cdot,a_2)](x))\cdot\nabla u(x)\geq\rho,\quad x\in B_r(\bar x).
\end{array}\end{equation}
Suppose in addition that there is $\delta>0$ such that if $x^1\in B_\delta(\bar x)$, and $x_p\in\partial\mathcal T$ is such that $d(x^1)=|x^1-x_p|$, then
\begin{equation}\label{eqdest1}
\nabla u(x_p)\cdot (f(x_p,a_1)+ f(x_p,a_2))\leq0.
\end{equation}
Therefore the minimum time function $T$ satisfies
\begin{equation}\label{eqfinalest}
T(x^1)\leq 2\sqrt{\frac{\tilde L}\rho}d(x^1)^{1/2}.\end{equation}
\end{cor}
\begin{proof}
Let $x^1\in B_\delta(\bar x)$, we just refine the proof of Theorem \ref{thm1} with the notations there used and the new assumption (\ref{eqdest1}). It implies that we can improve estimate (\ref{eq2nd}) since
$$\begin{array}{ll}
\nabla u(x^1)\cdot (f(x^1,a_1)+f(x^1,a_2))\\
=\nabla u(x^1)\cdot (f(x^1,a_1)+f(x^1,a_2))\pm\nabla u(x_p)\cdot (f(x_p,a_1)+f(x_p,a_2))
\leq  L_1d(x^1)
\end{array}$$
and then (\ref{eq2nd}) becomes
$$u(x^1_{2t})-u(\bar x)\leq L_1d(x^1)t-\frac{\rho}4t^2+Ld(x^1)\leq
\tilde Ld(x^1)-\frac{\rho}4t^2.$$
Now we can conclude exactly as before and at $x^1$ we obtain (\ref{eqfinalest}) instead.
\end{proof}
\begin{rem}\label{remextendedsym}
Note that, in order to check (\ref{eqdest1}), if there exist $\tilde a_1,\tilde a_2$ such that $f(x,\tilde a_1)\equiv -f(x,a_1)$ and $f(x,\tilde a_2)\equiv -f(x,a_2)$ for all $x\in B_r(\bar x)$, then exchanging the pair $(a_1,a_2)$ with $(\tilde a_1,\tilde a_2)$ we have that (\ref{eqcor}) is still satisfied, while we can change the sign of the left hand side of (\ref{eqdest1}). Therefore in this case we can have (\ref{eqdest1}) automatically satisfied. This occurs for instance for symmetric systems where $f(x,a)=\sigma(x)a$ and $A$ is symmetric to the origin. In that case (\ref{eqfinalest}) holds for all $x^1\in B_\delta(\bar x)$.

As an alternative to (\ref{eqdest1}), we could ask that
$$\nabla u(x^1)\cdot (f(x^1,a_1)+ f(x^1,a_2))\leq0,\quad x^1\in B_\delta(\bar x)
$$
and reach the same conclusion (\ref{eqfinalest}) as well, as immediately seen.
\end{rem}

The next result applies if the target ${\mathcal T}$ satisfies the sufficient conditions of Theorem \ref{thm1} at $\bar x$ and (\ref{eq1storder}) in Remark \ref{remcontrollability} locally on $\partial\mathcal T$ at the other points in the neighborhood of $\bar x$.
\begin{prop}\label{propdregular}
Suppose that for $\bar x\in\partial\mathcal T$ there are $r,C>0$, $r\leq1$ such that
$$T(x)\leq C|x-\bar x|^{1/2},\quad x\in B_r(\bar x).$$
Let $K,u\geq1$ and define
$$D:=(B_r(\bar x)\backslash\{\bar x\})\cap\{x:Kd(x)\leq |x-\bar x|^u\}.$$
Assume that for all $y\in B:=\partial\mathcal T\cap (B_r(\bar x)\backslash\{\bar x\})$ there are $r_y,\rho_y>0$ and $s\in(0,1]$ such that
\begin{equation}\label{eqestdtarget}
T(x)\leq \frac{2L}{\rho_y}d(x),\quad \hbox{for all }x\in B_{r_y}(y);\quad r_y^s\leq R\rho_y,\quad\hbox{and }D\subset\cup_{y\in B}B_{r_y}(y).
\end{equation}
Then for some $M\geq 0$ and $\bar s=\min\{1/(2u),1-s\}$
\begin{equation}\label{eqestest}
T(x)\leq M d^{\bar s}(x),\quad x\in B_r(\bar x).
\end{equation}
\end{prop}
\begin{proof}
Let $x\in B_r(\bar x)\backslash\mathcal T$.
If $x\notin D$ then
$$T(x)\leq C|x-\bar x|^{1/2}\leq CK^{1/(2u)}d^{1/(2u)}(x),$$
since $|x-\bar x|^u\leq K d(x)$.
If otherwise $x\in D$, let $y\in B$ such that $x\in B_{r_y}(y)$. Then by (\ref{eqestdtarget}),
$$T(x)\leq \frac{2L}{\rho_y}d(x)\leq \frac{2L}{\rho_y}r_y^sd^{1-s}(x)\leq 2LRd^{1-s}(x)$$
and we conclude.
\end{proof}
\begin{rem}
The condition (\ref{eqestdtarget}) on the set $B$ of Proposition \ref{propdregular} is about the rate of decay of the scalar product $\nabla u(x)\cdot f(x,a)$ as $x$ approaches $\bar x$. Indeed if at $\bar x$ the system satisfies a second order and not a first order controllability condition, then $r_y<|y-\bar x|$ and actually we can expect that to cover $D$ it is enough to suppose $\tilde Kr_y\leq |y-\bar x|^u$ for some $\tilde K>0$. Therefore the sufficient condition in the statement is satisfied if for all $y\in B$ there is some $a\in A$ such that $|y-\bar x|^{su}/R\leq -\nabla u(y)\cdot f(y,a)$, see Remark \ref{remcontrollability}.
If for instance $u=1,s=1/2$, then we reach (\ref{eqsecond}) with the exponent $1/2$. If instead $su=1$, $s=2/3$, then we reach (\ref{eqsecond}) with the exponent $1/3$. We racall that estimates of the form (\ref{eqestest}) are crucial to derive local H\"older continuity of the minimum time function and for solutions to boundary value problems for the Hamilton-Jacobi equation, see \cite{bcd,bfs}.
\end{rem}

\begin{rem}
One may notice that the partial differential equation corresponding to (\ref{eqpde}) recalls a stochastic control system. Suppose that indeed for some $a_1,a_2\in A$ we have
$$\begin{array}{ll}-{\mbox Tr}(D^2u\; f( x,a_1)\otimes f( x,a_2))
-(D(f( x,a_1)+f( x,a_2))(f( x,a_1)+f( x,a_2))\\
\quad+[f(\cdot,a_1),f(\cdot,a_2)]( x))\cdot\nabla u( x)\geq \rho>0,\quad x\in\R^n.
\end{array}$$
Consider then the stochastic system
$$\left\{\begin{array}{ll}
dx_t=(D(f(x_t,a_1)+f(x_t,a_2))(f(x_t,a_1)+f(x_t,a_2))+[f(\cdot,a_1),f(\cdot,a_2)](x_t))dt\\
\quad+\sqrt{2}(f(x_t,a_1)+f(x_t,a_2))\;dW_t,\\
x_0=x_o,
\end{array}\right.$$
where $W_t$ is a one dimensional, progressively measurable, brownian motion.
If $u\in C^2(\R^n)$, by Ito's formula we obtain the variation of $u$ on the trajectory of the system
$$\begin{array}{ll}
du_t=({\mbox Tr}(D^2u\; f( x_t,a_1)\otimes f( x_t,a_2))
+(D(f( x_t,a_1)+f( x_t,a_2))(f( x_t,a_1)+f( x_t,a_2))\\
+[f(\cdot,a_1),f(\cdot,a_2)]( x_t))\cdot\nabla u(x_t))\;dt+\sqrt{2}(f(x_t,a_1)+f(x_t,a_2))\cdot \nabla u(x_t)\;dW_t.
\end{array}$$
Thus by integrating on $(0,t)$, taking the expectation and using the fact that Ito's integral is a martingale we conclude that
$$\begin{array}{ll}
Eu(x_t)&=u(x_o)+E\int_0^t({\mbox Tr}(D^2u\; f( x_t,a_1)\otimes f( x_t,a_2))\\
&+(D(f( x_t,a_1)+f( x_t,a_2))(f( x_t,a_1)+f( x_t,a_2))
+[f(\cdot,a_1),f(\cdot,a_2)](x_tx))\cdot\nabla u(x_t))\;dt\\
&\leq u(x_o)-\rho t.
\end{array}$$
Therefore the mean value of $u$ decreases on the trajectories of the stochastic system.
\end{rem}
 
For completeness, we conclude this section recalling a necessary condition which is proved in \cite{bfs} for second order conditions. Recall first that when the target is the closure on an open set, it was proved in \cite{BaFa90} that the Lipschitz continuity of the minimum time function is characterised by the existence of  vector fields pointing inward the target. More precisely, if $\partial T$ is $C^2$ near $\bar x$, then 
\begin{equation}\label{T-Hol-as}
T(x) \leq C d(x)^s ,
\end{equation}
 near $\bar x$, with $s\in(1/2,1]$, if and only if there is $\bar a\in A$ such that $f(\bar x,\bar a) \cdot n(\bar x) < 0$, $n(\bar x)$ is the outward normal to the target. 
 
\begin{prop}\label{nec-cond-2}
Suppose that the target is defined as $\mathcal T=\{x:u(x)\leq u(\bar x)\}$ in the neighborhood of a point $\bar x\in \R^n$, where $u\in C^3(\R^n)$, $\nabla u(\bar x)\neq0$, $\nabla u(\bar x)\cdot f(\bar x,a)=0$ for all $a\in A$.  
Suppose that the controlled vector field $f\in C(\R^n\times A;\R^n)$ is of class $C^2$ 
in a neighborhood of $\bar x$ and $f(x,\cdot)$ is convex, for all $x$ in the neighborhood.
If for some $C,r>0$ and $s\in ]1/3, 1/2]$ 
\begin{equation}\label{eqnewsestimate}
T(x)\leq C|x-\bar x|^s,\quad \hbox{for }x\in B_r(\bar x),
\end{equation}
then either there are $a_1,a_2\in A$ such that
\begin{equation}
\label{trans-br1}
[f(\cdot,a_1),f(\cdot,a_2)]\cdot \nabla u(\bar x)<0,
\end{equation}
or
\begin{equation}\label{trans-br2}
\mbox{there is } \; \bar a\in A \; \text{such that}\quad\nabla(\nabla u\cdot f(\cdot,\bar a))\cdot f(\bar x,\bar a)<0 .
\end{equation}
\end{prop}

\subsection{Nonsmooth targets}
 
 In this section we quickly discuss how our sufficient condition can be applied also to some nonsmooth targets.
 As usual we assume that the target for system (\ref{eqsys}) is described, at least locally around a given point $\bar x\in\R^n$, as the closure of an open set satisfying (\ref{eqtarget}) where $u:\R^n\to\R$ is a continuous function such that the level set $\{x:u(x)=u(\bar x)\}$ has empty interior. Therefore we drop smoothness of the target but will suppose that there is an inward ball touching the boundary at the given point in the following sense.
We assume that
\begin{equation}\label{eqinward}
\mbox{there is a function }\Phi\in C^3(\R^n)\mbox{ such that }u-\Phi\mbox{ has a minimum point at }\bar x\mbox{ and }\nabla\Phi(\bar x)\neq0.\tag{IC}
\end{equation}

\begin{cor}\label{coric}
Let $f:\R^n\times A\to\R^{n}$ be of class $C^2$ and let $u:\R^n\to\R$ be a continuous function so that (\ref{eqtarget}) holds in the neighborhood of $\bar x$ and (\ref{eqinward}) holds. Assume moreover that $\nabla \Phi\cdot f(\bar x,a)=0$, for all $a\in A$ and that
\begin{equation}
\begin{array}{ll}
\max_{(a_1,a_2)\in A\times A}\{ -{\mbox Tr}(D^2\Phi\; f(\bar x,a_1)\otimes f(\bar x,a_2))\\
-((D(f(\bar x,a_1)+f(\bar x,a_2))(f(\bar x,a_1)+f(\bar x,a_2))+[f(\cdot,a_1),f(\cdot,a_2)](\bar x))\cdot\nabla \Phi(\bar x)\}>0.
\end{array}\end{equation}
Then the target $\{x:u(x)\leq u(\bar x)\}$ is STLA for the system (\ref{eqsys}) at $\bar x$ and the estimate (\ref{eqestimate}) holds for the minimum time function.
\end{cor}
\begin{proof}
Just observe that locally around $\bar x$ we have ${\mathcal T}_\Phi=\{x:\Phi(x)\leq\Phi(\bar x)\}\subset{\cal T}$, and ${\mathcal T}_\Phi$ is STLA at $\bar x$ by Theorem \ref{thm1}. Therefore $\mathcal T$ is STLA at $\bar x$.
\end{proof}

In the second example, we consider targets with outward corners by assuming that locally around a point $\bar x$ the target can be described as
\begin{equation}\label{eqoutward}
{\cal T}=\{x:u_i(x)\leq u_i(\bar x),\;i=1,\dots,k\}=\cap_{i=1}^k\{x:u_i(x)\leq u_i(\bar x)\},
\end{equation}
where $u_i\in C^3(\R^n)$ and $\nabla u_i(\bar x)\neq0$, $i=1,\dots,k$.

\begin{cor}\label{coroc}
Let $f:\R^n\times A\to\R^{n}$ be of class $C^2$ and let $u_i:\R^n\to\R$ be a family of functions of class $C^3$, $i=1,\dots,k$. 
Let $\bar x\in\R^n$ be a point with the following property: we can find $a_1,a_2\in A$ such that for each $i\in\{1,\dots,k\}$ either there are $\bar a\in A$, $\lambda>0$ such that
\begin{equation}\label{firstor}
\lambda f(\cdot,\bar a)\equiv f(\cdot,a_1)+f(\cdot,a_2),\quad
-f(\bar x,\bar a)\cdot \nabla u_i(\bar x)<0
\end{equation}
or $(f(\bar x,a_1)+f(\bar x,a_2))\cdot \nabla u_i(\bar x)=0$ and
\begin{equation}
\begin{array}{ll}
-{\mbox Tr}(D^2u\; f(\bar x,a_1)\otimes f(\bar x,a_2))\\
\quad -(D(f(\bar x,a_1)+f(\bar x,a_2))(f(\bar x,a_1)+f(\bar x,a_2))+[f(\cdot,a_1),f(\cdot,a_2)](\bar x)\cdot\nabla u_i(\bar x)\}>0.
\end{array}\end{equation}
Then the target ${\cal T}=\cap_{i=1,\dots,k}\{x:u_i(x)\leq u(\bar x)\}$ is STLA for the system (\ref{eqsys}) at $\bar x$.
\end{cor}
\begin{proof}
In view of Theorem \ref{thm1} and Remark \ref{remcontrollability} our assumptions allow to conclude that each set $\{x:u_i(x)\leq u(\bar x)\}$ is STLA at $\bar x$ for all $i=1,\dots,k$ and the same trajectory reaches each of them. Hence even their intersection is STLA at $\bar x$ and an estimate like (\ref{eqestimate}) holds for the minimum time function.
\end{proof}

\section{Symmetric systems}

In this section we want to discuss how Theorem \ref{thm1} applies to symmetric systems. We will find out that our sufficient condition is also necessary in this case and quite easy to check and can be reformulated as an algebraic condition. For simplicity we always work in the neighborhood of a given point $\bar x\in\partial \mathcal T$ and suppose that the target locally satisfies
\begin{equation}
\mathcal T=\{x:u(x)\leq u(\bar x)\},
\end{equation}
where $u\in C^1(\R^n)$ at least and $\nabla u(\bar x)\neq0$. 
In the case of symmetric systems the control vector field appears as $f(x,a)=\sigma(x)a$,
where $\sigma:\R^n\to\R^{n\times m}$ and $ a\in B_1(0)=\{a\in \R^m:|a|\leq1\}$.
For convenience we will indicate as $\sigma_i:\R^n\to\R^n$, $i=1,\dots,m$ the vector fields provided by the columns of the matrix valued $\sigma$. We will assume  at least $\sigma\in C^1(\R^n;\R^{n\times m})$. Therefore (\ref{eqsys}) becomes
\begin{equation}\label{eqsyssym}
\left\{\begin{array}{ll}\dot x_t=\sigma(x_t)a_t,\\
x_0\in {\mathbb R}^n,
\end{array}\right.
\end{equation}
where $a.:[0,+\infty)\to B_1(0)$, will always be piecewise constant.
The first lemma shows how we can rewrite the second order hamiltonians in the case of symmetric systems, introducing a bilinear form in the controls.
\begin{lem}\label{lemtech} Let $\sigma:\R^n\to\R^{n\times m}$ be of class $C^1$ and $u:\R^n\to\R$ be of class $C^2$. Let $a_1,a_2\in B_1(0)$ and $f=\sigma a_1,g=\sigma a_2$. 
Let $S:\R^n\to \R^{m\times m}$ be the continuous function $S(x)=D(\nabla u\;\sigma)\sigma(x)$.
Then:

\noindent
(i)  we can rewrite the second order hamiltonian as
$$H_f\circ H_gu(x)=H_{f,g}u(x)=S(x)\;a_1\cdot a_2,\quad a_1,a_2\in B_1(0).
$$
In particular $S_{i,j}(x)=H_{\sigma_j,\sigma_i}(x).$

\noindent
(ii) We can express the product with the Lie bracket
$$[f,g]\cdot\nabla u(x)=2S^e(x)a_1\cdot a_2,
$$
where $S^e$ denotes the skew symmetric part of $S$.
In particular
$$S^e(x)=\left(\frac12[\sigma_j,\sigma_i]\cdot \nabla u(x)\right)_{i,j=1,\dots,m}$$
and the matrix $S(x)$ is symmetric if and only if all Lie brackets among the vector fields $\sigma_i(x)$, $i=1,\dots,m$, are orthogonal to $\nabla u(x)$.

\noindent
(iii) The symmetric part of $S$ is 
$$\begin{array}{ll}
S^*(x)={}^t\sigma\;D^2u\;\sigma(x)
+\left(\frac12(D\sigma_j\;\sigma_i+D\sigma_i\;\sigma_j)\cdot\nabla u(x)\right)_{i,j=1,\dots,m}\end{array}.$$
\end{lem}
\begin{proof}
(i) It is just a simple computation
$$\begin{array}{ll}
H_{f,g}u(x)&=\nabla(\nabla u\; \sigma a_2)(x)\cdot \sigma(x) a_1\\&=
{}^tD(\nabla u\;\sigma)(x)a_2\cdot\sigma(x) a_1=D(\nabla u\;\sigma)(x)\sigma(x)\;a_1\cdot a_2.
\end{array}$$

\noindent
(ii) Again we compute, since the second order terms cancel out,
$$\begin{array}{ll}
[f,g]\cdot\nabla u(x)&=H_{f,g}u(x)-H_{g,f}u(x)\\
&=S(x)\;a_1\cdot a_2-S(x)\;a_2\cdot a_1=(S(x)-{}^tS(x))\;a_1\cdot a_2
=2S^e(x)a_1\cdot a_2.
\end{array}$$

\noindent
(iii) As easily seen in coordinates
$$\begin{array}{ll}
S(x)=D(\nabla u \;\sigma)\sigma(x)={}^t\sigma\;D^2u\;\sigma(x)
+(D\sigma_i\;\sigma_j\cdot\nabla u(x))_{i,j=1,\dots,m},
\end{array}$$
from which the symmetric part follows.
\end{proof}
We can rewrite the second order term in (\ref{eqest}) from which we derived our sufficient conditions in two ways, that are convenient in different ways, by using the matrix $S$. We introduce the matrix valued function $K:\R^n\to\R^{2m\times 2m}$,
\begin{equation}\label{eqkdef}
K(x)=\left(\begin{array}{cc}S^*(x)\quad&{}^tS(x)\\S(x)&S^*(x)
\end{array}\right).
\end{equation}
\begin{rem}
Notice that $K(x)$ is symmetric for any $x\in \R^n$. Moreover if
$\sigma:\R^n\to\R^{n\times m}$ is of class $C^1$ and $u:\R^n\to\R$ is of class $C^2$ then for all $a_1,a_2\in B_1(0)$ we get
\begin{equation}\label{lemtech2}
\begin{array}{ll}
K(x)\left(\begin{array}{c}a_1\\a_2\end{array}\right)\cdot \left(\begin{array}{c}a_1\\a_2\end{array}\right)
=S(x)a_1\cdot a_1+S(x)a_2\cdot a_2+2S(x)a_1\cdot a_2\\
=S^*(x)(a_1+a_2)\cdot(a_1+a_2)+2S^e(x)a_1\cdot a_2.
\end{array}\end{equation}
In view of Lemma \ref{lemtech} it is therefore clear that the quadratic form of $-K(\bar x)$ is what appears in (\ref{eqpde}) inside the max operation.
\end{rem}

We can now rephrase Proposition \ref{prop2} for symmetric systems.
\begin{prop}\label{prop3}
Let $t>0$ and $\sigma:\R^n\to\R^{n\times m}$ be of class $C^1$. Let $f=\sigma a_1$, $g=\sigma a_2$, for $a_1,a_2\in B_1(0)$ and $u:\R^n\to\R$ be a function of class $C^2$. The trajectory (\ref{eqswitch}) satisfies
\begin{equation}\label{eqestsym}\begin{array}{ll}
u(x_{2t})-u(x_0)=\nabla u\cdot\sigma(x_0)(a_1+a_2)t+K(x_0)\left(\begin{array}{c}a_1\\a_2\end{array}\right)\cdot \left(\begin{array}{c}a_1\\a_2\end{array}\right)\frac{t^2}2+o(t^2).
\end{array}\end{equation}
If in particular the vector fields $f,g$ are orthogonal to $\nabla u(x_0)$ at $x_0$, then
\begin{equation}\label{eqestsym2}\begin{array}{ll}
u(x_{2t})-u(x_0)=K(x_0)\left(\begin{array}{c}a_1\\a_2\end{array}\right)\cdot \left(\begin{array}{c}a_1\\a_2\end{array}\right)\frac{t^2}2+o(t^2).
\end{array}\end{equation}
\end{prop}

We proceed computing the minimum of the function ($K$ is as in (\ref{eqkdef})), for a fixed $\bar x\in\R^n$,
\begin{equation}\label{eqmin}
h(a_1,a_2)=K(\bar x)\left(\begin{array}{cc}a_1\\ a_2\end{array}\right)\cdot \left(\begin{array}{cc}a_1\\ a_2\end{array}\right),\quad |a_1|,|a_2|\leq 1,\end{equation}
in order to determine the best decrease rate of single switch trajectories for system (\ref{eqsys}) at a point $\bar x$ where all available vector fields are tangent to the level set of a given smooth function.
Then we characterise when it is strictly negative also through the properties of $S(\bar x)$. This requires some linear algebra.
\begin{prop}\label{propmatrk}
The matrix $K({\bar x})$ has $0$ as an eigenvalue. If (\ref{eqmin}) attains a negative minimum, then it is reached at an eigenvector $v=(a_1,a_2)$ of $K(\bar x)$ with minimal eigenvalue $\lambda$ and we have $|a_1|=|a_2|=1$, $h(v)=2\lambda$.
\end{prop}
\begin{proof}
It is clear that the minimum of $h$ in (\ref{eqmin}) is nonpositive, as by definition of $K$ in (\ref{eqkdef}), $h(a,-a)=0$ for all $a\in B_1(0)$, thus 0 is an eigenvalue of $K$. We will not write the dependence on $\bar x$ below. Also notice that the minimum of $h$ in the $\R^{2m}-$ball $B_{\sqrt{2}}((0,0))$, which contains $B_1(0)\times B_1(0)$, is attained at an eigenvector of norm $\sqrt{2}$ of the minimal eigenvalue of $K$.
We now show that if $(a_1,a_2)$ is an eigenvector of $K$ with non zero eigenvalue, then $|a_1|=|a_2|$. Therefore a negative minimum in (\ref{eqmin}) is also attained at an eigenvector of $K$ with norm $\sqrt{2}$ with minimal eigenvalue.
Let $(a_1,a_2)$ be an eigenvector of $K$ with $\lambda$ as an eigenvalue. Then it satisfies
\begin{equation}\label{eqsystem}
\left\{\begin{array}{ll}
S^*a_1+{}^tSa_2=\lambda a_1,\\
S^*a_2+Sa_1=\lambda a_2.
\end{array}\right.
\end{equation}
Multiply the first equation in (\ref{eqsystem}) by $a_2$ and the second by $a_1$. We obtain
\begin{equation}\label{eqsystem1}\left\{\begin{array}{ll}
S^*a_1\cdot a_2+Sa_2\cdot a_2=\lambda a_1\cdot a_2,\\
S^*a_2\cdot a_1+Sa_1\cdot a_1=\lambda a_1\cdot a_2,
\end{array}\right.\end{equation}
and then
\begin{equation}\label{eqss}
Sa_1\cdot a_1=Sa_2\cdot a_2.\end{equation}
Now restart from (\ref{eqsystem}) and multiply the first equation by $a_1$ and the second by $a_2$. We obtain
\begin{equation}\label{eqsystem2}\left\{\begin{array}{ll}
S^*a_1\cdot a_1+{}^tSa_2\cdot a_1=\lambda |a_1|^2,\\
S^*a_2\cdot a_2+Sa_1\cdot a_2=\lambda |a_2|^2,
\end{array}\right.\end{equation}
and therefore by (\ref{eqss})
$$\lambda(|a_1|^2-|a_2|^2)=0,$$
which gives us the conclusion.
\end{proof}
Putting things together we have the corresponding statement of Theorem \ref{thm1} for symmetric systems whose proof is now straightforward.
\begin{thm}\label{thm2}
Let $\sigma:\R^n\to \R^{n\times m}$ be of class $C^2$ and let $u:\R^n\to\R$ be a function of class $C^3$. Let $\bar x\in\R^n$ be a point such that $\nabla u(\bar x)\neq0$, $\nabla u\;\sigma(\bar x)=0$ and suppose that $K$ has a negative eigenvalue.
Then the target $\{x:u(x)\leq u(\bar x)\}$ is STLA for the symmetric system (\ref{eqsyssym}) at $\bar x$ and the minimum time function $T$ satisfies in the neighborhood of $\bar x$ the estimate (\ref{eqestimate}).
The coordinates of the eigenvector of $K(\bar x)$ with the minimal eigenvalue provide the controls for two vector fields that allow system (\ref{eqsyssym}) to reach the target in finite time.
\end{thm}

In the next result we characterise when the minimum in (\ref{eqmin}) is negative by properties of $S$.
\begin{thm}\label{thm3}
Let $\sigma:\R^n\to\R^{n\times m}$ be of class $C^2$ and let $u:\R^n\to\R$ be a function of class $C^3$. Let $\bar x\in\R^n$ be a point such that $\nabla u\;\sigma(\bar x)=0$. 
Then $K({\bar x})$ is positive semidefinite if and only if $S({\bar x})$ is symmetric and positive semidefinite.
In particular if $S({\bar x})$ is not symmetric and positive semidefinite, then $K({\bar x})$ has a negative eigenvalue and the system (\ref{eqsyssym}) is STLA at $\bar x$.
\end{thm}
\begin{proof} By (\ref{lemtech2}), it is clear that if $S({\bar x})$ is symmetric and positive semidefinite, then $K({\bar x})$ is positive semidefinite, as its skew symmetric part vanishes. We now prove the converse. Assume $K(\bar x)$ is positive semidefinite.

1. We suppose first that $S(\bar x)$ is symmetric. In the rest of the proof we will drop the dependence of the matrices on $\bar x$. Then again by (\ref{lemtech2})
$$K\left(\begin{array}{cc}a_1\\a_2\end{array}\right)\cdot \left(\begin{array}{cc}a_1\\a_2\end{array}\right)=S(a_1+a_2)\cdot (a_1+a_2).$$
Therefore if $K$ is positive semidefinite and we choose $a_1=a_2\in B_1(0)$, we get that $0\leq 4Sa_1\cdot a_1$, for all $a_1\in B_1(0)$, and $S$ is also positive semidefinite.

2. We suppose now that $S$ is not symmetric and show that $K$ must have a negative minimum on $B_1(0)\times B_1(0)$. In particular $S$ is not the null matrix. Consider the positive semidefinite matrix ${}^tSS$, it will have at least one positive eigenvalue $\lambda^2$ with corresponding unit eigenvector $a_1$. Thus 
$${}^tSSa_1=\lambda^2 a_1$$ 
and then
$$|Sa_1|^2=Sa_1\cdot Sa_1={}^tSSa_1\cdot a_1=\lambda^2 a_1\cdot a_1=\lambda^2,$$
so that $\lambda=|Sa_1|>0$. Just notice that if $\bar a$ is eigenvector of ${}^tSS$ with null eigenvalue, then the same argument shows that $S\bar a=0$.
If $a_1$ is therefore a unit eigenvector with $\lambda^2>0$ as an eigenvalue, let 
$$a_2=-\frac{S({\bar x})a_1}\lambda,$$
so that $|a_2|=1$. Now we obtain
$${}^tSa_2=-\lambda a_1,\quad Sa_1\cdot a_2=-\lambda,\quad
Sa_1\cdot a_1=-\lambda a_1\cdot a_2=Sa_2\cdot a_2.
$$
Thus we conclude that
$$K
\left(\begin{array}{cc}a_1\\ a_2\end{array}\right)\cdot \left(\begin{array}{cc}a_1\\ a_2\end{array}\right)=-2\lambda(1+a_1\cdot a_2).$$
We reach our conclusion that the left hand side is negative provided $a_1\neq -a_2$. 
Let us analyse this critical case. By definition it then follows
$$Sa_1=\lambda a_1,\quad {}^tSa_1=\lambda a_1.
$$
Therefore if this critical case happens for all eigenvectors of ${}^tSS$ with positive eigenvalues, and we consider an orthonormal basis of eigenvectors of ${}^tSS$, this is also a family of eigenvectors for $S$ which can then be diagonalised by an orthogonal matrix and is thus symmetric, which was supposed not to be the case.

The last conclusion now follows from Theorem \ref{thm2}.
\end{proof}

Finally the next statement shows a necessary and sufficient condition in order to have a second order STLA for a symmetric system at a point and (\ref{eqsecond}) satisfied.
\begin{prop} Consider the symmetric system (\ref{eqsyssym}) with $\sigma\in C^2(\R^n;\R^{n\times m})$.
Suppose that the target is defined as $\mathcal T=\{x:u(x)\leq u(\bar x)\}$ in the neighborhood of a point $\bar x\in \R^n$, where $u\in C^3(\R^n)$, $\nabla u(\bar x)\neq0$, $\nabla u(\bar x)\; \sigma(\bar x)=0$.  
The estimate \eqref{eqsecondb} holds
for all $x$ in a neighborhood of $\bar x$ if and only if the symmetric matrix $K(\bar x)$ has a negative eigenvalue or equivalently (\ref{eqpde}) holds.
Moreover in this case (\ref{eqsecond}) holds as well.
\end{prop}
\begin{proof} 
The necessary condition of Proposition \ref{nec-cond-2} shows that if (\ref{eqsecondb}) holds, two conclusions are possible. Either (\ref{trans-br1}) holds, in which case the matrix $S(\bar x)$ is not symmetric having by Lemma \ref{lemtech}(ii) a nontrivial skew symmetric part. Otherwise
$S(\bar x)$ is symmetric and (\ref{trans-br2}) holds, in which case $S(\bar x)$ will not be positive semidefinite by Lemma \ref{lemtech}(i). Therefore by Theorem \ref{thm3} $K(\bar x)$ is not positive semidefinite and must have a negative eigenvalue and thus (\ref{eqpde}) holds.

The fact that if $K(\bar x)$ has a negative eigenvalue, then the system is STLA at $\bar x$ is a direct consequence of Theorem \ref{thm2}. The fact that for a symmetric system (\ref{eqpde}) implies the estimate (\ref{eqsecond}) is a consequence of Corollary \ref{corsecondest} and Remark \ref{remextendedsym}.
\end{proof}

\section{Nonlinear affine systems}

The system is said to be nonlinear affine when the controlled vector field has the form $f(x,a)=\sigma_0(x)+\sigma(x)a$, where $\sigma\in C^1(\R^n;\R^{n\times m})$, $\sigma=(\sigma_1,\dots,\sigma_m)$, $a\in B_1(0)$, and $\sigma_0\in C^1(\R^n;\R^n)$.
Therefore (\ref{eqsys}) takes the form
\begin{equation}\label{eqsysnotsym}
\left\{\begin{array}{ll}\dot x_t=\sigma_0(x)+\sigma(x_t)a_t,\\
x_0\in {\mathbb R}^n,
\end{array}\right.
\end{equation}
where $a.:[0,+\infty)\to B_1(0)$ is piecewise constant.
We approach the discussion of our sufficient conditions for system (\ref{eqsysnotsym}) to reach the target $\mathcal T=\{x:u(x)\leq u(\bar x)\}$ in an algebraic way, as we did in the previous section for a symmetric one. We start extending the matrix valued function $\sigma$ by adding the vector field $\sigma_o$ as the first column. Then we construct the corresponding matrix $\tilde S$ of type $n\times(m+1)$ as in Lemma \ref{lemtech} so that it relates to the matrix $S$ of $\sigma$ as follows
$$\tilde S(x)=\left(\begin{array}{cc}
\alpha\quad&^t\beta\\
\gamma&S
\end{array}\right),$$
where $\alpha=H_{\sigma_0,\sigma_0}u$, $\beta=(H_{\sigma_j,\sigma_0}u)_{j=1,\dots,m}$, $\gamma=(H_{\sigma_0,\sigma_j}u)_{j=1,\dots,m}$. 
We also note for later use that
\begin{equation}\label{eqbetagamma}
\gamma-\beta=(H_{\sigma_0,\sigma_j}u-H_{\sigma_j,\sigma_0}u)_{j=1,\dots,m}=
([\sigma_0,\sigma_j]\cdot \nabla u)_{j=1,\dots,m}.
\end{equation}
Finally we introduce the corresponding matrix $\tilde K$ as in (\ref{eqkdef}). The analogue of Proposition \ref{prop1} now becomes as in the following statement.
\begin{prop}\label{propas}
Let $t>0$ and $\tilde\sigma:\R^n\to\R^{n\times (m+1)}$ be of class $C^1$. Let $f=\sigma_0+\sigma a_1$, $g=\sigma_0+\sigma a_2$, for $a_1,a_2\in B_1(0)\subset\R^m$, and $u:\R^n\to\R$ be a function of class $C^2$. The trajectory (\ref{eqswitch}) satisfies
\begin{equation}\label{eqestas}\begin{array}{ll}
u(x_{2t})-u(x_0)=\nabla u\cdot(2\sigma_0(x_0)+\sigma(x_0)(a_1+a_2))t+\tilde K(x_0)\left(\begin{array}{c}1\\a_1\\1\\a_2\end{array}\right)\cdot \left(\begin{array}{c}1\\a_1\\1\\a_2\end{array}\right)\frac{t^2}2+o(t^2),
\end{array}\end{equation}
as $t\to0$.
\end{prop}
It is clear that additional difficulties come from the fact that the second order term in (\ref{eqestas}) is not a nice quadratic form as before. In order to discuss its sign, we have to express it an a more readable way. First observe that, for $a_1,a_2\in B_1(0)$,
$$\tilde S(x_0)\left(\begin{array}{c}1\\a_1\end{array}\right)\cdot \left(\begin{array}{c}1\\a_2\end{array}\right)=\alpha+\beta\cdot a_1+\gamma\cdot a_2+S(x_0)a_1\cdot a_2.$$
Next we easily get another expression of the coefficient of the second order term in (\ref{eqestas}) as
\begin{equation}\label{eqfinalas}\begin{array}{ll}
k(a_1,a_2):=\tilde K(x_0)\left(\begin{array}{c}1\\a_1\\1\\a_2\end{array}\right)\cdot \left(\begin{array}{c}1\\a_1\\1\\a_2\end{array}\right)\\
\quad=\tilde S(x_0)\left(\begin{array}{c}1\\a_1\end{array}\right)\cdot \left(\begin{array}{c}1\\a_1\end{array}\right)+\tilde S(x_0)\left(\begin{array}{c}1\\a_2\end{array}\right)\cdot \left(\begin{array}{c}1\\a_2\end{array}\right)
+2\tilde S(x_0)\left(\begin{array}{c}1\\a_1\end{array}\right)\cdot \left(\begin{array}{c}1\\a_2\end{array}\right)
\\=4\alpha+(3\beta+\gamma)\cdot a_1+(\beta+3\gamma)\cdot a_2+
 K(x_0)\left(\begin{array}{c}a_1\\a_2\end{array}\right)\cdot \left(\begin{array}{c}a_1\\a_2\end{array}\right).
 \end{array}\end{equation}
All the trouble is created by $\alpha=\nabla(\nabla u\cdot \sigma_0)\cdot\sigma_0(x_0)$ and the other terms need to compensate for it. We obtain the following result which we write initially for $\alpha\leq0$ for simplicity, see Remark \ref{remalpha} for the general case. One further case will be considered in Corollary \ref{coronefield}.
\begin{thm}\label{thmnas}
Let $\sigma_0\in C^2(\R^n;\R^n)$, and $\sigma:\R^n\to \R^{n\times m}$ be of class $C^2$ and let $u:\R^n\to\R$ be a function of class $C^3$. Let $\bar x\in\R^n$ be a point such that $\nabla u(\bar x)\neq0$ and $\nabla u\;\tilde\sigma(\bar x)=0$.
Suppose moreover that $\nabla(\nabla\cdot \sigma_0)\cdot\sigma_0(\bar x)\leq0$ and either there is $j\in\{1,\dots,m\}$ such that $[\sigma_0,\sigma_j](x)\cdot\nabla u(\bar x)\neq 0$ or $K$ is not positive semidefinite.
Then the target $\{x:u(x)\leq u(\bar x)\}$ is STLA for the system (\ref{eqsys}) at $\bar x$ and the minimum time function $T$ satisfies in the neighborhood of $\bar x$ the estimate (\ref{eqestimate}).
\end{thm}
\begin{proof}
In view of (\ref{eqestas}) and the assumption, we have to analyse the sign of (\ref{eqfinalas}) at $\bar x=x_0$ in order to apply Theorem \ref{thm1}. We do it in several cases. We always assume $\alpha\leq0$.

If we choose $a_2=-a_1$ then by the definition of $K$,
\begin{equation}\label{eqkoper}
k(a_1,-a_1)=4\alpha+2(\beta-\gamma)\cdot a_1.
\end{equation}
\noindent
1. Suppose now that $[\sigma_0,\sigma_j](x)\cdot\nabla u(\bar x)\neq 0$ for some $j\in\{1,\dots,m\}$. Therefore $\gamma-\beta\neq0$. Let $a_1=(\gamma-\beta)/|\gamma-\beta|$, then \begin{equation}\label{eqcase2}
k(a_1,-a_1)=4\alpha-2|\gamma-\beta|=4\alpha -2|[\sigma_0,\sigma a_1]\cdot \nabla u(\bar x)|.
\end{equation}
Since this expression is negative, because $\alpha\leq0$, we can conclude.

In the next two cases we suppose that $K(\bar x)$ is not positive semidefinite.
We know by Theorem \ref{thm3} that then $S(\bar x)$ is either not symmetric or not positive semidefinite.

\noindent
2. Suppose first that $S^*(\bar x)$ has a negative eigenvalue.
Therefore there is $a_1\in B_1(0)$, $|a_1|=1$, $a_1$ eigenvector of $S^*(x_0)$ with corresponding eigenvalue $-\lambda<0$.
We now proceed by looking at the case $a_2=a_1$. We get
\begin{equation}\label{eqkoper1}
k(a_1,a_1)=4(\alpha+(\beta+\gamma)\cdot a_1+S(\bar x)a_1\cdot a_1).
\end{equation}
We can choose the direction of $a_1$ so that the mid term of the right hand side in (\ref{eqkoper1}) is nonpositive. We obtain
\begin{equation}\label{eqcase3}
k(a_1,a_1)\leq4(\alpha-\lambda),
\end{equation}
and we conclude, as $\alpha\leq0$.

\noindent
3. Suppose now that $S$ is not symmetric and in particular not null. Then we can find $i,j\in\{1,\dots,m\}$ such that $[\sigma_i,\sigma_j]\cdot\nabla u(x_0)\neq0$.
As in the proof of 
Theorem \ref{thm2}, 
 the positive semidefinite matrix $^tS(x_0)S(x_0)$ has a strictly positive eigenvalue $\lambda^2>0$ with corresponding unit eigenvector $a_1\in B_1(0)$. We will choose the direction of $a_1$ later. Then $\lambda=|S(x_0)a_1|>0$ and as before we choose $a_2=-S(x_0)a_1/\lambda$.
As in the proof of Theorem \ref{thm2} we obtain
\begin{equation}\label{eqnalast}
k(a_1,a_2)=4\alpha+((3\beta+\gamma)-\frac1\lambda\;^tS(x_0)(\beta+3\gamma))\cdot a_1-2\lambda(1+a_1\cdot a_2).
\end{equation}
We therefore choose the direction of $a_1$ so that the mid term in (\ref{eqnalast}) is not positive and get
\begin{equation}\label{eqcase4}
k(a_1,a_2)\leq 4\alpha-2\lambda(1+a_1\cdot a_2).
\end{equation}
We also know, by the proof of Theorem \ref{thm2}, that the second term in the right hand side of (\ref{eqcase4}) will be strictly negative for some choice of the eigenvector $a_1$, otherwise $S(\bar x)$ is symmetric. In our assumptions we can thus conclude.
\end{proof}
\begin{rem}\label{remalpha}
In the statement of Theorem \ref{thmnas} we only analysed the case $\alpha\leq0$. However as seen in the proof, if $\alpha>0$ this creates an obstacle to the local attainability of the target but other terms can compensate. In particular the conclusions of the theorem still hold if $4\alpha<2|\gamma-\beta|$ by (\ref{eqcase2}) in case 1; if $\alpha<\lambda0$ by (\ref{eqcase3}) in case 2; and finally if $2\alpha<\lambda(1+a_1\cdot a_2)$ by (\ref{eqcase4}) in case 3.
As in Remark \ref{remcontrollability}, it is enough to assume that there are $a_1,a_2\in B_1(0)$ such that $\nabla u(\bar x)\cdot(2\sigma_0(\bar x)+\sigma(\bar x)(a_1+a_2))=0$ instead of $\nabla u(\bar x)\tilde\sigma(\bar x)=0$, and that (\ref{eqpdepoint}) is satisfied in order to reach the same conclusions.
\end{rem}
One case is still missing in the previous statement, precisely the case of the single vector field allowing the system to reach the target. This is always contained in Theorem \ref{thm1} and we add it here for  immediate use.
\begin{cor}\label{coronefield}
Let $\sigma_0\in C^2(\R^n;\R^n)$, and $\sigma:\R^n\to \R^{n\times m}$ be of class $C^2$ and let $u:\R^n\to\R$ be a function of class $C^3$. Let $\bar x\in\R^n$ be a point such that $\nabla u(\bar x)\neq0$ and $\nabla u\;\tilde\sigma(\bar x)=0$.
Suppose that there is $a\in B_1(0)$ such that $\nabla(\nabla u\cdot(\sigma_0+\sigma a))\cdot(\sigma_0+\sigma(\bar x)a)<0$. Then
 the target $\{x:u(x)\leq u(\bar x)\}$ is STLA for the system (\ref{eqsys}) at $\bar x$ and the minimum time function $T$ satisfies in the neighborhood of $\bar x$ the estimate (\ref{eqestimate}).
\end{cor}

At this point we can revisit the necessary condition of Proposition \ref{nec-cond-2} for (\ref{eqnewsestimate}) in the case of nonlinear affine systems and show that it is also sufficient under a further condition on the drift.
\begin{prop}
Let $\sigma_0\in C^2(\R^n;\R^n)$, and $\sigma:\R^n\to \R^{n\times m}$ be of class $C^2$ and let $u:\R^n\to\R$ be a function of class $C^3$. Let $\bar x\in\R^n$ be a point such that $\nabla u(\bar x)\neq0$ and $\nabla u\;\tilde\sigma(\bar x)=0$. Suppose moreover that $\nabla(\nabla u\cdot\sigma_0)\cdot\sigma_0(\bar x)\leq0$. 
If the estimate (\ref{eqnewsestimate}) holds with $s\in(1/3,1/2]$, then either the sufficient condition in Theorem \ref{thmnas} or that in Corollary \ref{coronefield} must be satisfied.
\end{prop}
\begin{proof}
From Proposition \ref{nec-cond-2} we know that (\ref{eqnewsestimate}) implies either (\ref{trans-br1}) or (\ref{trans-br2}).
We start by supposing that (\ref{trans-br2}) is satisfied. Then that is precisely the sufficient condition in Corollary \ref{coronefield} for estimate (\ref{eqestimate}).

If instead (\ref{trans-br1}) holds, then there are $a_1,a_2\in B_1(0)$ such that
\begin{equation}\label{eqnecessaryna}
0>\nabla u(\bar x)\cdot [\sigma_0+\sigma a_1,\sigma_0+\sigma a_2](\bar x)=\nabla u(\bar x)\cdot ([\sigma_0,\sigma a_2](\bar x)+[\sigma a_1,\sigma_0](\bar x)+[\sigma a_1,\sigma a_2](\bar x)).
\end{equation}
Then one of the three terms in the right hand side of (\ref{eqnecessaryna}) is strictly negative.
We analyse them separately. If either $\nabla u(\bar x)\cdot [\sigma_0,\sigma a_2](\bar x)<0$ or $\nabla u(\bar x)\cdot [\sigma a_1,\sigma _0](\bar x)<0$, then the sufficient condition in case 1 in Theorem \ref{thmnas} is satisfied. If instead $\nabla u(\bar x)\cdot [\sigma a_1,\sigma a_2](\bar x)<0$, then the sufficient condition of case 3 is satisfied since $S(\bar x)$ is not symmetric. Hence (\ref{eqestimate}) is satisfied.
\end{proof}

\section{Examples}

\begin{exa}\label{ex1}
In this example we want to show that our condition for second order attainability can be satisfied by a single vector field. Consider the symmetric system
\begin{equation}
\left\{
\begin{array}{ll}
\dot x_t=-ay_t, \\
\dot y_t=ax_t  \\
(x_0,y_0)\in\R^2.  
\end{array}
\right.
\end{equation}
Here $a\in [-1,1]$. Let $u(x,y)=y-1$, $\sigma(x,y)={}^t(-y,x)$. Around $(x_0,y_0)=(0,1)$ we want to reach the target $\{(x,y):y\leq1\}$. Since $\nabla u(x,y)=(0,1)$ then $\nabla u(0,1)\sigma(0,1)=0$ and first order conditions do not apply. Instead we compute $S(x,y)=-y$, notice that it is scalar (symmetric) and negative for $y=1$. Indeed in this case
\begin{equation}\label{eqkk}
K=\left(\begin{array}{cc} -1\quad&-1\\-1&-1
\end{array}\right),\end{equation}
and $K$ has $(1,1)$ as an eigenvector of $-2$ as eigenvalue. Therefore the target is small time locally attainable at $(0,1)$ and we reach the target following control $a\equiv1$. 
\end{exa}

\begin{exa}\label{ex2}
(This example comes from \cite{ma2}). In $\R^2$, take $\sigma={}^t(0,1)$, $f(x,a)=\sigma a$ and $u(x,y)=\frac{1-x^2-y^2}2$ so there is a unique vector field which is constant. However $\nabla u\cdot\sigma(x,y)=-y$, therefore we have first order attainability of the $0-$sublevel set of $u$ unless $y=0$. At every point, in particular at $(1,0)$ we have $S=-1<0$ so there is second order attainability of $\{x:u(x,y)\leq0\}=\R^2\backslash B_1((0,0))$. Matrix $K$ is as in (\ref{eqkk}).
\end{exa}

\begin{exa}\label{ex3}
Consider in $\R^2$ the nonlinear affine system where
$$\sigma_0(x,y)=\left(
\begin{array}{cc}
y\\  0 
\end{array}
\right),\quad\sigma(x,y)=\left(
\begin{array}{cc}
 0\\ 1
\end{array}
\right),\quad u(x,y)=\frac{x^2+y^2}2,
$$
and one of its positive sublevel sets as a target.
Then $\nabla u\;\sigma (x,y)=(xy\quad y)$ which vanishes at points where $y=0$. So we consider $y=0$ and impose $x\neq0$, otherwise the gradient of $u$ vanishes, and look for second order conditions. Computing $\tilde S$ at such points we get
$$\tilde S(x,0)=\left(
\begin{array}{cc}
0\quad & x  \\
 0& 1
\end{array}\right).$$
Therefore in our notations $\alpha=0$, $\beta-\gamma=x$ and we know that the system satisfies an attainability condition of second order at $(x,0)$. By the proof of Theorem \ref{thmnas}, $a_1=-1$, $a_2=1$ are controls that lead the system to the target in one switch at $(x,0)$, for $x>0$.

Now consider the target $\mathcal T=\{(x,y):x^2+y^2\leq r^2\}$ and observe that $\nabla (x,y)\cdot(\sigma_0(x,y)\pm\sigma(x,y))=(x\pm1)y$.
Therefore if $r>1$ and if $(x,y)$ is a point of the boundary of the target in the neighborhood of $(r,0)$, when $y<0$ Petrov condition holds while for $y>0$ all available vector fields strictly point outward the target and therefore the target is not STLA at such points. Therefore a second order sufficient condition holds at $(r,0)$ but the target is not STLA at all points of its neighborhood. 
\end{exa}

\begin{exa}\label{ex4}(Heisenberg system)
In $\R^3$ consider the system where
$$\sigma(x,y,z)=\left(
\begin{array}{cc}
1\quad & 0  \\
 0& 1\\
 y&-x
\end{array}
\right),\quad u(x,y)=\frac{x^2+y^2+z^2}2.
$$
Then $\nabla u\;\sigma (x,y)={}^t(x+yz, y-xz)$ which vanishes at points where $x=y=0$, and we select $z\neq0$ because otherwise the gradient of $u$ vanishes. Computing $S$ at such points we get
$$S(0,0,z)=\left(
\begin{array}{cc}
1\quad & -z  \\
 z& 1
\end{array}\right).$$
which again is not symmetric and we know that the sublevel sets of $u$
are STLA for the system with a second order condition. In this case we computed the minimal eigenvalue at $z=1$ which has multiplicity 2
$$K(0,0,1)=\left(
\begin{array}{cccc}
1\quad & 0  \quad &1\quad &-1\\
 0& 1&1&1\\
1&1&1&0\\
-1&1&0&1
\end{array}
\right),
$$
and $\lambda_{min}=1-\sqrt{2}<0$. Eigenvectors providing the highest decrease rate of $u$ are $(\frac{\sqrt{2}}2,-\frac{\sqrt{2}}2,0,1)$, $(-\frac{\sqrt{2}}2,-\frac{\sqrt{2}}2,1,0)$ and the vector space generated by them. Each of the two pairs of coordinates, e.g. $(\frac{\sqrt{2}}2,-\frac{\sqrt{2}}2),(0,1)$, give us controls to determine the two vector fields that we need to use to achieve attainability of the sublevel sets of $u$ with maximal rate among trajectories with at most one switch.
\end{exa}

\begin{exa}\label{ex5}(Convexified Reeds Shepp system)
In $\R^3$ take the symmetric system where
$$\sigma(x,y,z)=\left(
\begin{array}{cc}
\cos{z}\quad & 0  \\
 \sin{z}& 0\\
 0&1
\end{array}
\right),\quad u(x,y)=\frac{x^2+y^2+z^2}2.
$$
Therefore $\nabla u\;\sigma (x,y,z)={}^t(x\cos{z}+y\sin{z}, z)$ which vanishes at points where $x=z=0$, and we add $y\neq0$ because otherwise the gradient of $u$ vanishes. Computing $S$ at such points we get
$$S(0,y,0)=\left(
\begin{array}{cc}
1\quad & y  \\
 0& 1
\end{array}\right).$$
which is not symmetric so the sublevel sets of $u$ are STLA around $(0,y,0)$, $y\neq0$.
\end{exa}

\begin{exa}\label{ex6}
This example is taken from \cite{ma}. Consider the nonlinear affine system in $\R^3$ where
$$\sigma_0(x,y,z)\left(\begin{array}{c}
-y/12\\x/12\\0\end{array}\right),\quad\sigma(x,y,z)=\left(\begin{array}{cc}
xz\quad&0\\yz&0\\0&1\end{array}\right),$$
and the target is a positive sublevel set of the function $u(x,y,z)=(x^2+y^2)/2$. Then 
$$\nabla u(x,y,z)\tilde\sigma(x,y,z)=(0,(x^2+y^2)z,0).$$
The only points where a Petrov condition is not satisfied are on the plane $z=0$, and we add $x^2+y^2>0$ to avoid the singular $z-$axis.
At such points, we can compute
$$\tilde S(x,y,0)=\left(\begin{array}{ccc}0\quad&0\quad&0\\0&2(x^2+y^2)z^2&x^2+y^2\\0&0&0\end{array}\right).$$
Then $\alpha=0$, $\beta=0=\gamma$ but $S(x,y,0)$ is not symmetric and therefore the system satisfies a second order controllability condition and the sublevel set of $u$ is STLA at $(x,y,0)$.
\end{exa}



\begin{thebibliography}{99}
 
\bibitem{ba}
Bacciotti, A., {\it Processus de Controle avec Controle Initial, Proceedings of the Conference on Recent Advances in Differential Equations}, Trieste, Italy, 1978, Edited by R. Conti, Academic Press, New York, New York, pp. 23–36, 1981.

\bibitem{bcd}
 Bardi, Martino; Capuzzo-Dolcetta, Italo, Optimal control and viscosity solutions of Hamilton-Jacobi-Bellman equations. With appendices by Maurizio Falcone and Pierpaolo Soravia. Systems \& Control: Foundations \& Applications. Birkh\"auser Boston, Inc., Boston, MA, 1997.

\bibitem{BaFa90}
Bardi, M.; Falcone, M. {\it An approximation scheme for the minimum time function}, SIAM J. Control Optim. 28 (1990), no. 4, 950–965.

\bibitem{baso2}
Bardi, Martino; Soravia, Pierpaolo, {\it Hamilton-Jacobi equations with singular boundary conditions on a free boundary and applications to differential games}, Trans. Amer. Math. Soc. 325 (1991), no. 1, 205–229.

\bibitem{bfs} Bardi, Martino; Feleqi, Ermal; Soravia, Pierpaolo, {\it Regularity of the minimum time and of viscosity solutions of degenerate eikonal equations via generalized Lie brackets}, Arxiv no. 1907.02399.

\bibitem{bs}
Bianchini, R. M., and Stefani, G., {\it Local controllability for Analytic Families of Vector Fields}, Rapporto No. 19, Istituto Matematico U. Dini, Firenze, Italy, 1981/82.

\bibitem{bs2}
Bianchini, R. M., and Stefani, G., {\it Stati Localmente Controllabili}, Rendiconti del Seminario Matematico Universitario del Politecnico di Torino, Vol. 42, No. 1, pp. 15–23, 1984.

\bibitem{bs3} Bianchini, R. M. ;  Stefani, G.  {\it Time-optimal problem and time-optimal map},
 Rend. Sem. Mat. Univ. Politec. Torino  48  (1990),  no. 3, 401--429.

\bibitem{kr}
Krastanov, Mikhail; Quincampoix, Marc, {\it Local small time controllability and attainability of a set for nonlinear control system,}ESAIM Control Optim. Calc. Var. 6 (2001), 499–516.

\bibitem{kr2}
Krastanov, Mikhail Ivanov, {\it 
A sufficient condition for small-time local controllability,}
SIAM J. Control Optim. 48 (2009), no. 4, 2296–2322. 

\bibitem{ma3}
  Le, Thuy T. T.; Marigonda, Antonio, {\it Small-time local attainability for a class of control systems with state constraints}, ESAIM Control Optim. Calc. Var. 23 (2017), no. 3, 1003–1021.
 
\bibitem{li}
Liverovskii, A. A., {\it Hölder's Conditions for Bellman's Functions}, Differentsialnye Uravneniia, Vol. 13, No. 12, pp. 2180–2187, 1977.

\bibitem{li2}
Liverovskii, A. A., {\it Some Properties of Bellman's Functions for Linear Symmetric Polisystems}, Differentsialnye Uravneniia, Vol. 16, No. 3, pp. 414–423, 1980.

\bibitem{ma}
 Marigonda, Antonio, {\it Second order conditions for the controllability of nonlinear systems with drift,} Commun. Pure Appl. Anal. 5 (2006), no. 4, 861–885.
 
\bibitem{ma2}
 Marigonda, Antonio; Rigo, Silvia, {\it Controllability of some nonlinear systems with drift via generalized curvature properties.} SIAM J. Control Optim. 53 (2015), no. 1, 434–474.
 
 \bibitem{mr}
Motta, Monica; Rampazzo, Franco, {\it Asymptotic controllability and Lyapunov-like functions determined by Lie brackets,} SIAM J. Control Optim. 56 (2018), no. 2, 1508–1534.
 
 \bibitem{pe}
Petrov, N. N., {\it On the Bellman Function for the Time-Optimal Process Problem}, Prikladnaia Matematika i Mechanika, Vol. 34, No. 5, pp. 820–826, 1970.

\bibitem{pe2}
Petrov, N. N., {\it Controllability of Autonomous Systems}, Differentsialnye Uravneniia, Vol. 4, No. 4, pp. 606–617, 1968.

\bibitem{so}
Soravia, P., {\it 
H\"older continuity of the minimum-time function for $C^1$-manifold targets.}
J. Optim. Theory Appl. 75 (1992), no. 2, 401–421.

\bibitem{so1}
Soravia, P., {\it Some results on second order controllability conditions}, Proceedings of the 58th IEEE Conference on Decision and Control CDC 2019, to appear.

\end{thebibliography}
\end{document}